\documentclass{amsart}
\usepackage{inputenc, amsmath,amssymb,graphicx,verbatim, verbatim, bbm, scalerel}

\usepackage[hidelinks]{hyperref}
\usepackage{amsthm,amsrefs}
\usepackage{tikz-cd}
\usepackage{tikz}
\usetikzlibrary{calc}
\usepackage{mathtools}
\usepackage{soul}
\usepackage{multirow}
\usepackage[margin=1.5in]{geometry}
\newtheorem{theorem}{Theorem}

\newtheorem{lem}{Lemma}

\newtheorem{prop}[theorem]{Proposition}
\newtheorem{cor}[lem]{Corollary}

\DeclareMathOperator*{\E}{\scalerel*{\mathbb{E}}{\sum}}

\title[Landau's Prime Ideal Theorem]{A New Elementary Proof of Landau's Prime Ideal Theorem, and Associated Results}
\author{Alex Burgin}
\address{School of Mathematics, Georgia Institute of Technology, Atlanta, GA}
\email{alexander.burgin@gatech.edu}
\date{\today}

\begin{document}

\begin{abstract}
 We give a new elementary proof of Landau's prime ideal theorem. 
 The proof is an extension of Richter's proof of the prime number theorem.  The main result contains other results related to the equidistribution of the prime ideal counting function.  

\end{abstract}

\maketitle

\tableofcontents

\section{Introduction}

We give a new elementary proof of Landau's prime ideal theorem \cite{landau}, which is in the spirit of Richter's recent elementary proof of the prime number theorem \cite{Richter_2021}.  
That argument is phrased in the language of `orthogonality' of arithmetic functions.  

Let $K$ be a number field. Unique factorization of elements in $K$ can  fail, but holds in the ring of Integral Ideals $I_K$. We can then speak of prime ideals $\mathfrak p\in I_K$.  For a general ideal $\mathfrak m$, let $\Omega(\mathfrak m) $ denote the number of prime ideals dividing $\mathfrak m$, counting multiplicity. By unique factorization of $I_K$, such a function is well-defined. Elements $\mathfrak m \in I_K$ also have a norm $N(\mathfrak m)$. Our main theorem is as follows: 

\begin{theorem}\label{mainTheorem}
    Let $K$ be a number field. Then there exists a positive integer $N$ such that, for any $g:\mathbb{N}_0\rightarrow \mathbb{C}$ satisfying $|g (n)|\leq 1$ for all $n$, one has
    \begin{align}
        \sum_{\mathfrak m \colon N(\mathfrak{m})\leq x}g(\Omega(\mathfrak{m})+k_1)=\sum_{\mathfrak m \colon N(\mathfrak{m})\leq x}g(\Omega(\mathfrak{m})+k_2)+o_K\Big(\sum_{\mathfrak m\colon N(\mathfrak{m})\leq x}1\Big)
    \end{align} for each $k_1,k_2\geq N$. 
\end{theorem}

This contains several results. 
For Landau's theorem, take $g (x) = (-1)^x$, $k_1$ even, and $k_2$ odd. We see that asymptotically, the number of ideals $\mathfrak m$ 
for which $\Omega (\mathfrak m)$ is even is equal to those for for which $\Omega(\mathfrak m)$ is odd.  That is an equivalent form of Landau's theorem, 
see Lemma \ref{LiouvilleImpliesPi} and Appendix C. 

Let $ e(y) = e^{2\pi i y}$ for real $y$. 
For integer $q >2$, set  $g (y)=e(ay/q)$, for each $1\leq a  < q$. 
It then follows that by choosing $k_1\geq N$ and $k_2=k_1+1$ that
\begin{align*}
    (1-e(a/q))e(ak_1/q)\sum_{N(\mathfrak{m})\leq x}e(a\Omega(\mathfrak{m})/q)
    &=o_K\Big(\sum_{N(\mathfrak{m})\leq x}1\Big), 
    \end{align*} and hence $\sum_{N(\mathfrak{m})\leq x}e(a\Omega(\mathfrak{m})/q)=o_{K}\Big(\sum_{N(\mathfrak{m})\leq x}1\Big)$ for each $1\leq a<q$. This extends a theorem of Pillai \cite{pillai} and Selberg \cite{selbergErgodic},  establishing the equidistribution of $(\Omega(\mathfrak{m}))_{N(\mathfrak{m})\rightarrow \infty}$ modulo $q$ for each $q\geq 1$. See Appendix B for a concise explanation.

Take  $g(x)=e(\alpha x)$ for $\alpha \in \mathbb{R}\setminus \mathbb{Q}$.  With  choices of $k_1$ and $k_2$ depending upon $\alpha$, we can verify that 
\begin{align*}
     \sum_{N(\mathfrak{m})\leq x}e(\alpha \Omega(\mathfrak{m}))
     & = o_K\Big(\sum_{N(\mathfrak{m})\leq x}1\Big). 
\end{align*}
 This is a  theorem of Erd\H{o}s \cite{erdosAdditive} and Delange \cite{delange}, generalized to the number field setting. 
 Namely, that the fractional parts  $(\{\Omega(\mathfrak{m})
 \alpha\} \colon N(\mathfrak{m})\leq x )$ are asymptotically uniformly distributed in $[0,1)$ for any irrational $\alpha$.

\subsection{Notation, Conventions, and Standard Definitions}

We use $\mathcal{O}_K$ to denote the ring of integers of the number field $K$. Within $\mathcal{O}_K$, we let $I_K$ denote the set of integral ideals, and denote ideals in $I_K$ with $\mathfrak{m},\mathfrak{n}$, and letters in similar font. We use $\mathfrak{p}$ or $\mathfrak{q}$ to denote a prime ideal, and define $N_{K/\mathbb{Q}}:I_K\rightarrow \mathbb{Z}$ to be the ideal norm taking an ideal $\mathfrak{m}$ to its norm $N_{K/\mathbb{Q}}(\mathfrak{m})$. Usually, we will simply drop the subscript and just write $N(\mathfrak{m})$. $I_K$ is naturally endowed with a ring structure via ideal addition and multiplication, and $N:I_K\rightarrow \mathbb{Z}$ is a homomorphism from $(I_K,\cdot)\rightarrow (\mathbb{Z},\cdot)$.

Given some $m\in \mathbb{N}$, we say an integer sequence $(a_n)_{n\geq 0}$ is equidistributed modulo $m$ if \begin{align*}
    \lim_{x\rightarrow \infty}\frac{1}{x}\sum_{\substack{0\leq n<x\\ a_n\equiv c\ (\text{mod }m)}}1=\frac{1}{m}\quad \quad \text{for all }c\in \mathbb{Z}/m\mathbb{Z}.
\end{align*} Equidistribution of an integer sequence is closely related to the behavior of associated exponential sums; for more explanation, see Appendix B.\\

When we write $\sum_{N(\mathfrak{m})\leq x}$, we mean that we are summing over all ideals $\mathfrak{m}$ with norm no greater than $x$.\\

For a finite set of ideals $S\subset I_K$ and a function $f: I_K\rightarrow \mathbb{C}$, we use $\E_{\mathfrak{s}\in S}f(\mathfrak{s})$ to denote the standard averaging operator \begin{align*}
    \E_{\mathfrak{s}\in S}f(\mathfrak{s}):=\frac{1}{|S|}\sum_{\mathfrak{s}\in S}f(\mathfrak{s})
\end{align*} and $\E_{\mathfrak{s}\in S}^{\log}f(\mathfrak{s})$ to denote the logarithmic (weighted) averaging operator \begin{align*}
    \E_{\mathfrak{s}\in S}^{\log }f(\mathfrak{s}):=\frac{1}{\sum_{\mathfrak{s}\in S}\frac{1}{N(\mathfrak{s})}}\sum_{\mathfrak{s}\in S}\frac{f(\mathfrak{s})}{N(\mathfrak{s})}.
\end{align*} These are analogues of the operators in \cite{Richter_2021}.\\

We use the standard big-Oh notation $f(x)=O(g(x))$ if $|f(x)|\leq C|g(x)|$ for a constant $C$; similarly, $f(x)=O_{\alpha,\beta,...}(g(x))$ reads that $|f(x)|\leq C|g(x)|$ for some $C=C(\alpha,\beta,...)$. We use $f(x)=o(g(x))$ if the quotient $|f(x)/g(x)|\rightarrow 0$ as $x\rightarrow \infty$.

\subsection{Background}

The distribution of the prime-omega function over number fields $K$ is intimately tied to the distribution of prime numbers, echoing the classical case $K=\mathbb{Q}$. Recall the following lemma:

\begin{lem}\label{LiouvilleImpliesPi}
    Suppose that $L(x):=\sum_{N(\mathfrak{m})\leq x}(-1)^{\Omega(\mathfrak{m})}=o_K(x)$, i.e., that $\Omega$ is equidistributed modulo 2. Then $\pi_K(x):=\sum_{N(\mathfrak{p})\leq x}1\sim \frac{x}{\log x}$.
\end{lem}

An elementary proof of this lemma (which is actually an equivalence, but the proof of that will not be needed) follows from proving that the associated function $M(x):=\sum_{N(\mathfrak{m})<x}\mu(\mathfrak{m})=o_K(x)$, where $$\mu(\mathfrak{m})=\begin{cases}
    0 & \text{if there exists a nonunit ideal }\mathfrak{p}\text{ such that }\mathfrak{p}^2|\mathfrak{m}\\
    (-1)^{\Omega(\mathfrak{m})} & \text{otherwise.}
\end{cases}$$ This function is analogous to the so-called Mertens function over the case $K=\mathbb{Q}$, and showing sublinear growth for $M(x)$ implies that $\pi_K(x)\sim \frac{x}{\log x}$. This is a well known implication for the case of number fields, following from tight ideal density bounds (see \S 1.3 and \cite{diamondZhangEquivalences}). We will leave this implication unproven, but provide an elementary proof in Appendix C that $L(x)=o_K(x)\implies M(x)=o_K(x)$. Then, with Theorem 1, this recovers Landau's prime ideal theorem.\\

\subsection{Density Constants}

A well-known and elementary\footnote{To see an explanation as to why this is true, see \cite[Ch. 6: The Distribution of Ideals in a Number Ring]{marcus}.} fact (in that it requires no complex analysis) is that for any number field $K$ of degree $d$, one has $\sum_{N(\mathfrak{m})\leq x}1=c_{K}x+O_K(x^{1-\frac{1}{d}})$ for a constant $c_K> 0$. This constant $c_K$ is often referred to as the `ideal density' of $\mathcal{O}_K$ (see, e.g., \cite{montgomeryVaughan}, \S8.4).  The power saving-nature of the error term is important, but not essential–we note that these methods can be extended to derive PNTs and equidistribution results for many multiplicative semigroups (or `Beurling number systems') with suitably-quickly decaying error term in density.

\section{Building Towards Theorem 1}

The following result is fundamental to our paper: \begin{prop}\label{orthogonalityProp}
    Let $K$ be a number field of degree $d$, and let $c_K$ be the ideal density of $\mathcal{O}_K$. Then for any $S\subset I_K$ finite and nonempty, \begin{align}\label{prop2Statement}
        \frac{1}{x}\sum_{N(\mathfrak{m})\leq x}\Big|\sum_{\mathfrak{n}\in S}\mathbf{1}_{\mathfrak{n}|\mathfrak{m}}-\sum_{\mathfrak{n}\in S}\frac{1}{N(\mathfrak{n})}\Big|^2=c_K\sum_{\mathfrak{n}_1,\mathfrak{n}_2\in S}\frac{\Phi(\mathfrak{n}_1,\mathfrak{n}_2)}{N(\mathfrak{n}_1)N(\mathfrak{n}_2)}+O_K(|S|^{1+\frac{1}{d}}\cdot x^{-\frac{1}{d}})
    \end{align} where $\Phi(\mathfrak{n}_1,\mathfrak{n}_2):=N(\operatorname{gcd}(\mathfrak{n}_1,\mathfrak{n}_2))-1$.
\end{prop}

\begin{proof}
    Let $A$ denote the sum $\sum_{\mathfrak{n}\in S}\frac{1}{N(\mathfrak{\mathfrak{n}})}$; then by expanding the square on the left hand side  of \eqref{prop2Statement} one has \begin{align*}
        \Big|\sum_{\mathfrak{n}\in S}\mathbf{1}_{\mathfrak{n}|\mathfrak{m}}-A\Big|^2 &=\Big(\sum_{\mathfrak{n}\in S}\mathbf{1}_{\mathfrak{n}|\mathfrak{m}}\Big)^2-2A\sum_{\mathfrak{n}\in S}\mathbf{1}_{\mathfrak{n}|\mathfrak{m}}+A^2 ,
    \end{align*} and summing this over all $\mathfrak{m}$ with $N(\mathfrak{m})\leq x$ and dividing by $x$, one has \begin{align*}
        \frac{1}{x}\sum_{N(\mathfrak{m})\leq x}\Big|\sum_{\mathfrak{n}\in S}
        &\mathbf{1}_{\mathfrak{n}|\mathfrak{m}}-A\Big|^2  \coloneqq S_1-2AS_2+S_3, 
    \\ 
    \noalign{\noindent where} 
        S_1&:=\frac{1}{x}\sum_{N(\mathfrak{m})\leq x}\Big(\sum_{\mathfrak{n}\in S}\mathbf{1}_{\mathfrak{n}|\mathfrak{m}}\Big)^2 ,\\ S_2&:=\frac{1}{x}\sum_{N(\mathfrak{m})\leq x}\sum_{\mathfrak{n}\in S}\mathbf{1}_{\mathfrak{n}|\mathfrak{m}},
        \\ S_3&:= \frac{A^2}{x}\sum_{N(\mathfrak{m})\leq x}1.
    \end{align*} 
    
    \emph{Calculation of $S_3$}. Using that $\sum_{N(\mathfrak{m})\leq x}1=c_Kx+O_K(x^{1-\frac{1}{d}})$ we see that $S_3=c_KA^2+O_K(A^2x^{-\frac{1}{d}})$. Then, using that $A=O_K(\log |S|)$ (see Appendix A for a method of proof), we have $S_3=c_KA^2+O_K(\log^2(|S|)x^{-\frac{1}{d}})$.\\

    \emph{Calculation of $S_2$}. For any fixed $\mathfrak{n}$, if $\mathfrak{n}|\mathfrak{m}$ then one can write $\mathfrak{m}=\mathfrak{n}\mathfrak{m}'$ in a unique way, and hence \begin{align*}\frac{1}{x}\sum_{\mathfrak m \colon  N(\mathfrak{m})\leq x}\mathbf{1}_{\mathfrak{n}|\mathfrak{m}}&=\frac{1}{x}\sum_{\mathfrak m \colon  N(\mathfrak{nm'})\leq x}1 \\ &= \frac{1}{x}\sum_{\mathfrak m \colon N(\mathfrak{m}')\leq x/N(\mathfrak{n})}1 \\ &=\frac{c_K}{N(\mathfrak{n})}+O_K\Big(x^{-\frac{1}{d}}N(\mathbf{n})^{\frac{1}{d}-1}\Big)\end{align*} so that one has, after summing over $\mathfrak{n}\in S$, \begin{align*}
        S_2=\sum_{\mathfrak{n}\in S}\frac{c_K}{N(\mathfrak{\mathfrak{n}})}+O_K\Big(x^{-\frac{1}{d}}\sum_{\mathfrak{n}\in S}\frac{1}{N(\mathfrak{n})^{1-\frac{1}{d}}}\Big)=c_KA+O_K\Big(x^{-\frac{1}{d}}\sum_{\mathfrak{n}\in S}\frac{1}{N(\mathfrak{n})^{1-\frac{1}{d}}}\Big).
    \end{align*} Then since \begin{align}\label{error1}
        \sum_{\mathfrak{n}\in S}\frac{1}{N(\mathfrak{n})^{1-\frac{1}{d}}}\ll_K |S|^{\frac{1}{d}}
    \end{align} (for a proof of \eqref{error1}, see Appendix A), we then have $S_2=c_KA+O_K(x^{-\frac{1}{d}}|S|^{\frac{1}{d}})$.\\
    
    \emph{Calculation of $S_1$}. We begin by noting that we can write 
    \begin{equation}
        S_1=\frac{1}{x}\sum_{\mathfrak m \colon N(\mathfrak{m})\leq x}\sum_{\mathfrak{n}_1,\mathfrak{n}_2\in S}\mathbf{1}_{\mathfrak{n}_1|\mathfrak{m}}\mathbf{1}_{\mathfrak{n}_2|\mathfrak{m}}
    \end{equation}
    For the moment, we view $\mathfrak n_1, \mathfrak n_2$ as fixed.     
    Then, we note that since for any $\mathfrak{m}$ with $\mathfrak{n}_1|\mathfrak{m}$ we can write $\mathfrak{m}=\mathfrak{m}'\mathfrak{n}_1$ in a unique way, so that \begin{align}\label{eq5}\frac{1}{x}\sum_{\mathfrak m \colon N(\mathfrak{m})\leq x}\mathbf{1}_{\mathfrak{n}_1|\mathfrak{m}}\mathbf{1}_{\mathfrak{n}_2|\mathfrak{m}}&=\frac{1}{x}\sum_{\mathfrak m' \colon N(\mathfrak{m'n}_1)\leq x}\mathbf{1}_{\mathfrak{n}_2|\mathfrak{m'n}_1} \\ &= \nonumber \frac{1}{x}\sum_{\mathfrak m' \colon N(\mathfrak{m}')\leq x/N(\mathfrak{n}_1)} \mathbf{1}_{\mathfrak{n}_2|\mathfrak{m}'\mathfrak{n}_1} \end{align} and then letting $\mathfrak{d}=\operatorname{gcd}(\mathfrak{n}_1,\mathfrak{n}_2)$, $\mathfrak{d}\mathfrak{d}_i=\mathfrak{n}_i$ ($i=1,2$) we have $\mathfrak{n}_2|\mathfrak{m}'\mathfrak{n}_1$ if and only if $\mathfrak{d}_2|\mathfrak{m}'\mathfrak{d}_1$, and hence $\mathbf{1}_{\mathfrak{n}_2|\mathfrak{m}'\mathfrak{n}_1}=\mathbf{1}_{\mathfrak{d}_2|\mathfrak{m'}\mathfrak{d}_1}$. Then since $\mathfrak{d}_2,\mathfrak{d}_1$ are coprime, if $\mathfrak{d}_2|\mathfrak{m}'\mathfrak{d}_1$ we must have $\mathfrak{d}_2|\mathfrak{m}'$ so that, from (\ref{eq5}), 
    \begin{align*}
        \frac{1}{x}\sum_{\mathfrak m \colon N(\mathfrak{m})\leq x}\mathbf{1}_{\mathfrak{n}_1|m}\mathbf{1}_{\mathfrak{n}_2|m}&= \frac{1}{x}\sum_{\mathfrak m' \colon N(\mathfrak{m}')\leq x/N(\mathfrak{n}_1)}\mathbf{1}_{\mathfrak{d}_2|\mathfrak{m}'} \\ &= \frac{1}{x}\sum_{\mathfrak m''  \colon N(\mathfrak{m}'')\leq x/N(\mathfrak{n}_1\mathfrak{d}_2)}1 \\ &= \frac{c_K}{N(\mathfrak{n}_1\mathfrak{d}_2)}+O_K(x^{-\frac{1}{d}}N(\mathfrak{n}_1\mathfrak{d}_2)^{\frac{1}{d}-1}) \\ &= \frac{c_K\cdot N(\operatorname{gcd}(\mathfrak{n}_1,\mathfrak{n}_2))}{N(\mathfrak{n}_1)N(\mathfrak{n}_2)}+O_K\Big(x^{-\frac{1}{d}}\frac{N(\operatorname{gcd}(\mathfrak{n}_1,\mathfrak{n}_2))^{1-\frac{1}{d}}}{N(\mathfrak{n}_1)^{1-\frac{1}{d}}N(\mathfrak{n}_2)^{1-\frac{1}{d}}}\Big).
    \end{align*}

Defining $\Phi(\mathfrak{n}_1,\mathfrak{n}_2):=N(\operatorname{gcd}(\mathfrak{n}_1,\mathfrak{n}_2))-1$ and then summing over $\mathfrak{n}_1,\mathfrak{n}_2\in S$ one then has \begin{align*}
        S_1&=\sum_{\mathfrak{n}_1,\mathfrak{n}_2\in S}\frac{c_K\Phi(\mathfrak{n}_1,\mathfrak{n}_2)}{N(\mathfrak{n}_1)N(\mathfrak{n}_2)}+\sum_{\mathfrak{n}_1,\mathfrak{n}_2\in S}\frac{c_K}{N(\mathfrak{n}_1)N(\mathfrak{n}_2)}+E_1 \\ &=\sum_{\mathfrak{n}_1,\mathfrak{n}_2\in S}\frac{c_K\Phi(\mathfrak{n}_1,\mathfrak{n}_2)}{N(\mathfrak{n}_1)N(\mathfrak{n}_2)}+c_KA^2+E_1,
    \end{align*} 
    where, using \eqref{error1}, 
    \begin{align*}
        E_1 & \ll_K x^{-\frac{1}{d}}\sum_{\mathfrak{n}_1,\mathfrak{n}_2\in S}\Big(\frac{N(\operatorname{gcd}(\mathfrak{n}_1,\mathfrak{n}_2))}{N(\mathfrak{n}_1)N(\mathfrak{n}_2)}\Big)^{1-\frac{1}{d} } 
        \\ 
 &   \ll_K x^{-\frac{1}{d}}\sum_{\mathfrak{n}_1,\mathfrak{n}_2\in S}\frac{1}{N(\mathfrak{n}_2)^{1-\frac{1}{d}}}\ll x^{-\frac{1}{d}}|S|^{1+\frac{1}{d}}. 
    \end{align*} 
    We then have 
    \begin{align*}
        S_1=\sum_{\mathfrak{n}_1,\mathfrak{n}_2\in S}\frac{c_K\Phi(\mathfrak{n}_1,\mathfrak{n}_2)}{N(\mathfrak{n}_1)N(\mathfrak{n}_2)}+c_KA^2+O_K(x^{-\frac{1}{d}}|S|^{1+\frac{1}{d}}).\\
    \end{align*}

    Now, let us denote $\sum_{\mathfrak{n}_1,\mathfrak{n}_2\in S}\frac{c_K\Phi(\mathfrak{n}_1,\mathfrak{n}_2)}{N(\mathfrak{n}_1)N(\mathfrak{n}_2)}$ with $S_\Phi$. One has \begin{align*}
        S_1-2AS_2+S_3 &= \Big(S_\Phi +c_KA^2+O_K(x^{-\frac{1}{d}}|S|^{1+\frac{1}{d}})\Big)\\ &\quad -2A\Big(c_KA+O_K(x^{-\frac{1}{d}}|S|^{\frac{1}{d}})\Big) \\ &\quad+ \Big(c_KA^2+O_K(\log^2(|S|)x^{-\frac{1}{d}})\Big)
        \\ 
        &=S_\Phi+O_K(x^{-\frac{1}{d}}|S|^{1+\frac{1}{d}})
    \end{align*} since the $A^2$ term cancels. 
    This is the  desired estimate. 
\end{proof}

We then want an equivalent formulation to Proposition 1 in terms of expectation:

\begin{cor}[Equivalent Assertion to Proposition \ref{orthogonalityProp}]
    Let $K$ be a number field of degree $d$, and let $c_K$ be the ideal density of $\mathcal{O}_K$. Suppose $S\subset K$ is finite and nonempty. Then \begin{align}
        \E_{N(\mathfrak{m})\leq x}\Big|\E_{\mathfrak{n}\in S}^{\log{}}(N(\mathfrak{n})\mathbf{1}_{\mathfrak{n}|\mathfrak{m}}-1)\Big|^2=\E_{\mathfrak{n}\in S}^{\log{}}\E_{\mathfrak{n}'\in S}^{\log{}}\Phi(\mathfrak{n},\mathfrak{n}')+O_K(x^{-\frac{1}{d}}|S|^{1+\frac{1}{d}}).
    \end{align} 
\end{cor}

\begin{proof} We can begin by writing Proposition 2 as $$\sum_{N(\mathfrak{m})\leq x}\Big|\sum_{\mathfrak{n}\in S}\mathbf{1}_{\mathfrak{n}|\mathfrak{m}}-A\Big|^2=c_Kx\sum_{\mathfrak{n}_1,\mathfrak{n}_2\in S}\frac{\Phi(\mathfrak{n}_1,\mathfrak{n}_2)}{N(\mathfrak{n}_1)N(\mathfrak{n}_2)}+O_K(x^{1-\frac{1}{d}}|S|^{1+\frac{1}{d}}).$$ Dividing both sides by $\sum_{N(\mathfrak{m})\leq x}1=c_Kx+O_K(x^{-\frac{1}{d}})$ gives then that \begin{align*}
    \E_{N(\mathfrak{m})\leq x}\Big|\sum_{\mathfrak{n}\in S}\mathbf{1}_{\mathfrak{n}|\mathfrak{m}}-A\Big|^2&=\frac{1}{1+O_K(x^{-\frac{1}{d}})}\sum_{\mathfrak{n}_1,\mathfrak{n}_2\in S}\frac{\Phi(\mathfrak{n}_1,\mathfrak{n}_2)}{N(\mathfrak{n}_1)N(\mathfrak{n}_2)}+O_K(x^{-\frac{1}{d}}|S|^{1+\frac{1}{d}}) \\ &= \sum_{\mathfrak{n}_1,\mathfrak{n}_2\in S}\frac{\Phi(\mathfrak{n}_1,\mathfrak{n}_2)}{N(\mathfrak{n}_1)N(\mathfrak{n}_2)}+O_K(x^{-\frac{1}{d}}|S|^{1+\frac{1}{d}}).
\end{align*} Dividing again by $A^2$ and noting that $\frac{1}{A}\sum_{\mathfrak{n}\in S}\mathbf{1}_{\mathfrak{n}|\mathfrak{m}}=\E_{\mathfrak{n}\in S}^{\log{}}N(\mathfrak{n})\mathbf{1}_{\mathfrak{n}|\mathfrak{m}}$, one has \begin{align*}
    \E_{N(\mathfrak{m})\leq x}\Big|\E_{\mathfrak{n}\in S}^{\log{}}N(\mathfrak{n})\mathbf{1}_{\mathfrak{n}|\mathfrak{m}}-1\Big|^2=\E_{\mathfrak{n}_1\in S}^{\log{}}\E_{\mathfrak{n}_2\in S}^{\log } \Phi(\mathfrak{n}_1,\mathfrak{n}_2)+O_K(x^{-\frac{1}{d}}|S|^{1+\frac{1}{d}}).
\end{align*}

\end{proof}

The following proposition is our main ingredient, analogous to Richter's Proposition 2.2:
\begin{prop}\label{prop2.2analogue}
    For all $\eta>0$ there exists some $N=N(\eta)\in \mathbb{N}$ such that for all $k\geq N$ there exist two nonempty sets of ideals $S_1=S_1(\eta,k),S_2=S_2(\eta,k)\subset I_K$ for which the following are true: \begin{enumerate}
        \item[(i)] All ideals in $S_1$ are prime and all ideals in $S_2$ are a product of exactly $k$ prime ideals
        \item[(ii)] The sets $S_1$ and $S_2$ have the same cardinality, and there exists an enumeration $S_1=\{\mathfrak{p}_{1},\mathfrak{p}_2,...,\mathfrak{p}_j\}$ and $S_2=\{\mathfrak{m}_1,\mathfrak{m}_2,...,\mathfrak{m}_j\}$ for which $(1-\eta)N(\mathfrak{p}_i)\leq N(\mathfrak{m}_i)\leq (1+\eta)N(\mathfrak{p}_i)$ for all $i=1,...j$.
        \item[(iii)] $\E^{\log}_{\mathfrak{n}\in S_i}\E^{\log}_{\mathfrak{n}'\in S_i}\Phi(\mathfrak{n},\mathfrak{n}')\leq \eta$ for $i=1,2$.
    \end{enumerate}
\end{prop}

\begin{proof}[Proof of Theorem \ref{mainTheorem} assuming Proposition \ref{prop2.2analogue}] Fix $\eta>0$ and take $g:\mathbb{N}_0\rightarrow \mathbb{C}$ with $|g|\leq 1$. Choose $S_1=S_1(\eta)$ and $S_2=S_2(\eta)$ to the desired specifications above. Note first that \begin{align*}
    \E_{\mathfrak{m}\in S_i}^{\log}\E_{N(\mathfrak{n})\leq x/N(\mathfrak{m})}g(\Omega(\mathfrak{m}\mathfrak{n}))&=\E_{\mathfrak{m}\in S_i}^{\log}\frac{1}{\sum_{N(\mathfrak{n})\leq x/N(\mathfrak{m})}1}\sum_{N(\mathfrak{n})\leq x/N(\mathfrak{m})}g(\Omega(\mathfrak{mn}))\\ &= \E_{\mathfrak{m}\in S_i}^{\log}\frac{1}{c_K\cdot \frac{x}{N(\mathfrak{m})}+O_K(x^{1-\frac{1}{d}})}\sum_{N(\mathfrak{n}')\leq x}1_{\mathfrak{m}|\mathfrak{n}'}g(\Omega(\mathfrak{n}')) \\ &= \E_{\mathfrak{m}\in S_i}^{\log}(1+O_K(x^{-\frac{1}{d}}))\E_{N(\mathfrak{n'})\leq x}N(\mathfrak{m})1_{\mathfrak{m}|\mathfrak{n}'}g(\Omega(\mathfrak{n}')) \\ &= \E_{\mathfrak{m}\in S_i}^{\log} \E_{N(\mathfrak{n}')\leq x}N(\mathfrak{m})1_{\mathfrak{m}|\mathfrak{n}'}g(\Omega(\mathfrak{n}'))+O_{S_i,K}(x^{-\frac{1}{d}}).
\end{align*} Now we notice that \begin{align*}
    &\Big|\E_{N(\mathfrak{n})\leq x}g(\Omega(\mathfrak{n}))-\E_{\mathfrak{m}\in S_2}^{\log}\E_{N(\mathfrak{n})\leq x/N(\mathfrak{m})}g(\Omega(\mathfrak{m}\mathfrak{n}))\Big|^2\\ &\leq \Big|\E_{N(\mathfrak{n})\leq x}g(\Omega(\mathfrak{n}))-\E_{\mathfrak{m}\in S_2}^{\log}\E_{N(\mathfrak{n}')\leq x}N(\mathfrak{m})1_{\mathfrak{m}|\mathfrak{n}'}g(\Omega(\mathfrak{n}'))\Big|^2+O_{S_2,K}(x^{-\frac{1}{d}})\\ &\leq \E_{N(\mathfrak{n})\leq x}\Big|\E_{\mathfrak{m}\in S_2}^{\log}(1-N(\mathfrak{m})1_{\mathfrak{m}|\mathfrak{n}})\Big|^2+O_{S_2,K}(x^{-\frac{1}{d}}) \\ &\leq \eta+O_{S_2,K}(x^{-\frac{1}{d}}),
\end{align*} where we have used the bound $|g|\leq 1$ extensively alongside Proposition \ref{orthogonalityProp}, the Cauchy-Schwarz inequality, and the assumption on $S_2$. Since $\Omega(\mathfrak{mn})=\Omega(\mathfrak{m})+\Omega(\mathfrak{n})$ we have \begin{align}\label{eq6}
    \E_{N(\mathfrak{m})\leq x}g(\Omega(\mathfrak{m}))&=\E_{\mathfrak{m}\in S_2}^{\log}\E_{N(\mathfrak{n})\leq x/N(\mathfrak{m})}g(\Omega(\mathfrak{m})+\Omega(\mathfrak{n}))+O_K(\eta^{\frac{1}{2}})+O_{S_2,K}(x^{-\frac{1}{2d}}) \\ \nonumber &= \E_{\mathfrak{m}\in S_2}^{\log}\E_{N(\mathfrak{n})\leq x/N(\mathfrak{m})}g(k+\Omega(\mathfrak{n}))+O_K(\eta^{\frac{1}{2}})+O_{S_2,K}(x^{-\frac{1}{2d}}).
\end{align} A similar argument with $S_1$ instead of $S_2$ and $g(n+k-1)$ in place of $g(n)$ gives \begin{align}\label{eq7}
    \E_{N(\mathfrak{m})\leq x}g(\Omega(\mathfrak{m})+k-1)&=\E_{\mathfrak{p}\in S_1}^{\log{}}\E_{N(\mathfrak{n})\leq x/N(\mathfrak{p})}g(\Omega(\mathfrak{p})+\Omega(\mathfrak{n})+k-1)+O_K(\eta^{\frac{1}{2}})+O_{S_1,K}(x^{-\frac{1}{2d}}) \\ \nonumber &=\E_{\mathfrak{p}\in S_1}^{\log{}}\E_{N(\mathfrak{n})\leq x/N(\mathfrak{p})}g(k+\Omega(\mathfrak{n}))+O_K(\eta^{\frac{1}{2}})+O_{S_1,K}(x^{-\frac{1}{2d}}). 
\end{align} Finally, we know that there exists an enumeration of $S_1,S_2$ such that $S_1=\{\mathfrak{p}_{1},\mathfrak{p}_2,...,\mathfrak{p}_j\}$ and $S_2=\{\mathfrak{m}_1,\mathfrak{m}_2,...,\mathfrak{m}_j\}$ and $(1-\eta)N(\mathfrak{p}_i)\leq N(\mathfrak{m}_i)\leq (1+\eta)N(\mathfrak{p}_i)$ for all $i=1,...j$. Then it easily follows that $\E_{N(\mathfrak{n})\leq x/N(\mathfrak{m}_i)}g(k+\Omega(\mathfrak{n}))=\E_{N(\mathfrak{n})\leq x/N(\mathfrak{p}_i)}g(k+\Omega(\mathfrak{n}))+O_K(\eta)$. Taking logarithmic averages over $S_1$ and $S_2$ yields \begin{align}\label{eq8}
    \E_{\mathfrak{m}\in S_2}^{\log{}}\E_{N(\mathfrak{\mathfrak{n}})\leq x/N(\mathfrak{m})}g(k+\Omega(\mathfrak{n}))=\E_{\mathfrak{p}\in S_1}^{\log{}}\E_{N(\mathfrak{n})\leq x/N(\mathfrak{p})}g(k+\Omega(\mathfrak{n}))+O_K(\eta).
\end{align} From applying (\ref{eq8}) to (\ref{eq6}) and then applying (\ref{eq7}) we then deduce that \begin{align*}
\E_{N(\mathfrak{m})\leq x}g(\Omega(\mathfrak{m}))=\E_{N(\mathfrak{m})\leq x}g(\Omega(\mathfrak{m})+k-1)+O_K(\eta^{\frac{1}{2}})+O_{S_1,S_2,K}(x^{-\frac{1}{2d}}).
\end{align*} This holds for all $k\geq N$, so that one then has \begin{align}
    \E_{N(\mathfrak{m})\leq x}g(\Omega(\mathfrak{m})+k_1)=\E_{N(\mathfrak{m})\leq x}g(\Omega(\mathfrak{m})+k_2)+O_K(\eta^{\frac{1}{2}})+O_{S_1,S_2,K}(x^{-\frac{1}{2d}})
\end{align} for all $k_1,k_2\geq N$. The theorem then follows upon taking $x\rightarrow \infty$; since $\eta>0$ was arbitrary, we are finished.

\end{proof}

\section{Proof of Proposition \ref{prop2.2analogue} } 

\subsection{Chebychev-Type Bounds for $\pi_K(x)$}

We have this upper bound for prime ideals: \begin{lem}\label{UpperBoundLemma}
    Let $\alpha\geq 1$ and $x\geq e$. One then has \begin{align*}
        \sum_{\mathfrak p \colon x<N(\mathfrak{p})\leq \alpha x}1\leq (\alpha \log \alpha) \frac{x}{\log x}+O_{\alpha,K}(x^{1-\frac{1}{d}}\log x).
    \end{align*}
\end{lem}

\begin{proof}
     The classical way to derive weak-type prime bounds is by studying the prime factorizations of binomial coefficients; in particular, the central binomial coefficient $\binom{2n}{n}=\frac{(2n)!}{(n!)^2}$ is at least as large as the product of primes $n < p < 2n$.\\

In the number field setting, we instead follow this line of reasoning. Define $f(x):=\prod_{N(\mathfrak{m})\leq x}N(\mathfrak{m})$. One can observe that $\log f(x)=c_Kx\log x-c_Kx+O_K(x^{1-\frac{1}{d}}\log x)$ (by using Abel summation and ideal density arguments; refer to Proposition \ref{AppendixAProp} in Appendix A for a more general argument). Then it follows, if we further define $\theta_\alpha(x):=f(\alpha x)/f(x)^\alpha$, that \begin{align}\label{thetaFunctionAsympt}
    \log \theta_\alpha(x)=c_K(\alpha \log \alpha)x+O_{\alpha,K}(x^{1-1/d}\log x).
\end{align} At the same time, by using multiplicity of the norm, one has that $f(x)=N\Big(\prod_{N(\mathfrak{m})\leq x}\mathfrak{m}\Big)$. Appealing to unique factorization, we consider the factorization of the ideal $\prod_{N(\mathfrak{m})\leq x}\mathfrak{m}$. Let $\mathfrak{p}$ be a prime ideal. We know that there are a total of $\frac{c_Kx}{N(\mathfrak{p})}+O_K(x^{1-\frac{1}{d}}N(\mathfrak{p})^{-1})$ ideals divisible by $\mathfrak{p}$ over the set of ideals with norm at most $x$, and $\frac{c_Kx}{N(\mathfrak{p})^2}+O_K(x^{1-\frac{1}{d}}N(\mathfrak{p})^{-2})$ divisible by $\mathfrak{p}^2$, and so on. 
    Then, 
     the multiplicity $\nu_{\mathfrak p} $ of a prime ideal $\mathfrak{p}$ in the ideal $\prod_{N(\mathfrak{m})\leq x}\mathfrak{m}$ is given by 
     \begin{align*}
\nu_\mathfrak{p}\Big(\prod_{N(\mathfrak{m})\leq x}\mathfrak{m}\Big)
    &=   
c_Kx\sum_{c=1}^{\log_{N(\mathfrak{p})}(x)}N(\mathfrak{p})^{-c}+O_K\Big(x^{1-\frac{1}{d}}\sum_{c=1}^{\log_{N(\mathfrak{p})}(x)}N(\mathfrak{p})^{-c}\Big) \\ &=c_Kx\sum_{c=1}^{\log_{N(\mathfrak{p})}(x)}N(\mathfrak{p})^{-c}+O_K\Big(\frac{x^{1-\frac{1}{d}}\log x}{N(\mathfrak{p})\log N(\mathfrak{p})}\Big).
\end{align*} 

One can then estimate \begin{align*}
   \log f(x)&= \sum_{N(\mathfrak{p})\leq x}\nu_\mathfrak{p}\Big(\prod_{N(\mathfrak{m})\leq x}\mathfrak{m}\Big)\cdot \log N(\mathfrak{p})\\ &=c_Kx\sum_{N(\mathfrak{p})\leq x}\log N(\mathfrak{p})\sum_{c=1}^{\log_{N(\mathfrak{p})}(x)}N(\mathfrak{p})^{-c}+O_K\Big(x^{1-\frac{1}{d}}\log x \sum_{N(\mathfrak{p})\leq x}\frac{1}{N(\mathfrak{p})}\Big) \\ &= c_Kx\sum_{N(\mathfrak{p})\leq x}\log N(\mathfrak{p})\sum_{c=1}^{\log_{N(\mathfrak{p})}(x)}N(\mathfrak{p})^{-c}+O_K\Big(x^{1-\frac{1}{d}}(\log x)^2\Big).
\end{align*}

To deduce the last line above, we used that $\sum_{N(\mathfrak{p})\leq x}\frac{1}{N(\mathfrak{p})}\leq \sum_{N(\mathfrak{m})\leq x}\frac{1}{N(\mathfrak{m})}$ and then applied Proposition \ref{AppendixAProp} in Appendix A\footnote{It is worth noting that one actually has the better bound $\sum_{N(\mathfrak{p})\leq x}\frac{1}{N(\mathfrak{p})}\ll_K \log \log x$, which one can get quite immediately from following Merten's original proof of the result in $\mathbb{Q}$ (see \cite{mertens}) in this more general setting; but this is immaterial for our purposes.}. We now estimate the function $\theta_\alpha(x)$ using this factorization estimate.  One can compute \begin{align} \label{numEq}
        \log f(\alpha x)&=\alpha c_Kx\sum_{N(\mathfrak{p})\leq \alpha x}\log N(\mathfrak{p})\sum_{c=1}^{\log_{N(\mathfrak{p})}(\alpha x)}N(\mathfrak{p})^{-c}+O_{\alpha,K}(x^{1-\frac{1}{d}}(\log x)^2), \\ \alpha \log f(x)&= \label{denomEq} \alpha c_Kx\sum_{N(\mathfrak{p})\leq x}\log N(\mathfrak{p})\sum_{c=1}^{\log_{N(\mathfrak{p})}(x)}N(\mathfrak{p})^{-c}+O_{\alpha,K}(x^{1-\frac{1}{d}}(\log x)^2)
    \end{align} so that using the relation $\log \theta_\alpha(x)=\log f(\alpha x)-\alpha \log f(x)$, one observes that \begin{align}\label{logThetaEquality}
        \log \theta_\alpha(x)= \alpha c_Kx\sum_{N(\mathfrak{p})\leq \alpha x}\log N(\mathfrak{p})\sum_{c=1+\lfloor \log_{N(\mathfrak{p})}(x)\rfloor}^{\log_{N(\mathfrak{p})}(\alpha x)}N(\mathfrak{p})^{-c}+O_{\alpha,K}(x^{1-\frac{1}{d}}(\log x)^2).
    \end{align} Truncating the sum to $x<N(\mathfrak{p})\leq \alpha x$ one surely has \begin{align*}
        \log \theta_\alpha(x)&\geq \alpha c_Kx\sum_{x<N(\mathfrak{p})\leq \alpha x}\log N(\mathfrak{p})\sum_{c=1}^{\log_{N(\mathfrak{p})}(\alpha x)}N(\mathfrak{p})^{-c}+O_{\alpha,K}(x^{1-\frac{1}{d}}(\log x)^2)\end{align*} and since the inner sum contributes at least the term $c=1$, \begin{align*}&\geq \alpha c_Kx\sum_{x<N(\mathfrak{p})\leq \alpha x}\frac{\log N(\mathfrak{p})}{N(\mathfrak{p})}+O_{\alpha,K}(x^{1-\frac{1}{d}}(\log x)^2) \\ &\geq \alpha c_Kx\frac{\log (\alpha x)}{\alpha x}\sum_{x<N(\mathfrak{p})\leq \alpha x}1+O_{\alpha,K}(x^{1-\frac{1}{d}}(\log x)^2) \\ &\geq c_K\log x\sum_{x<N(\mathfrak{p})\leq \alpha x}1+O_{\alpha,K}(x^{1-\frac{1}{d}}(\log x)^2),
    \end{align*} assuming $x\geq e$. Hence by applying (\ref{thetaFunctionAsympt}) one then has \begin{align*}
        \sum_{x<N(\mathfrak{p})\leq \alpha x}1\leq (\alpha \log \alpha )\frac{x}{\log x}+O_{\alpha,K}(x^{1-\frac{1}{d}}\log x).
    \end{align*} This completes the proof of the lemma.
\end{proof}

\begin{lem}\label{lowerBoundLemma}
    Let $\pi_K(y)$ be the number of prime ideals in $\mathcal{O}_K$ with norm at most $y$. Then one has that \begin{align*}
        \pi_K(y)\geq \frac{y}{e\log y}+O_{K}(y^{1-1/d}\log y).
    \end{align*}
\end{lem}

\begin{proof}
    Recall that, from (\ref{logThetaEquality}), one has that \begin{align*}
        \log \theta_\alpha(x)=\alpha c_Kx\sum_{N(\mathfrak{p})\leq \alpha x}\log N(\mathfrak{p})\sum_{c=1+\lfloor \log_{N(\mathfrak{p})}(x)\rfloor}^{\log_{N(\mathfrak{p})}(\alpha x)}N(\mathfrak{p})^{-c}+O_{\alpha,K}(x^{1-\frac{1}{d}}(\log x)^2).
    \end{align*} The inner sum consists of $\lfloor \log_{N(\mathfrak{p})}(\alpha x)\rfloor -\lfloor \log_{N(\mathfrak{p})}(x)\rfloor$ terms, with each term being at most \begin{align*}
        N(\mathfrak{p})^{-(1+\lfloor \log_{N(\mathfrak{p})}(x)\rfloor)}\leq N(\mathfrak{p})^{-\log_{N(\mathfrak{p})}(x)}=1/x
    \end{align*} in magnitude. Hence \begin{align*}
        \log \theta_\alpha (x)&\leq \alpha c_K\sum_{N(\mathfrak{p})\leq \alpha x}\log N(\mathfrak{p})\Big(\lfloor\log_{N(\mathfrak{p})}(\alpha x)\rfloor -\lfloor \log_{N(\mathfrak{p})}(x)\rfloor\Big)+O_{\alpha,K}(x^{1-\frac{1}{d}}(\log x)^2) \\ &\leq \alpha c_K \sum_{N(\mathfrak{p})\leq \alpha x}\log N(\mathfrak{p})\lfloor \log_{N(\mathfrak{p})}(\alpha x)\rfloor +O_{\alpha,K}(x^{1-\frac{1}{d}}(\log x)^2) \\ &\leq \alpha c_K\log(\alpha x)\sum_{N(\mathfrak{p})\leq \alpha x}1+O_{\alpha,K}(x^{1-\frac{1}{d}}(\log x)^2).
    \end{align*} Then since $\log \theta_\alpha (x)=c_K(\alpha \log \alpha )x+O_{\alpha,K}(x^{1-1/d}\log x)$, we produce that \begin{align*}
        \sum_{N(\mathfrak{p})\leq \alpha x}1\geq \frac{(\log \alpha)x}{\log(\alpha x)}+O_{\alpha,K}(x^{1-1/d}\log x).
    \end{align*} Setting $y=\alpha x$ then gives that \begin{align*}
        \sum_{N(\mathfrak{p})\leq y}1\geq \frac{\log \alpha}{\alpha}\frac{y}{\log y}+O_{\alpha,K}(y^{1-1/d}\log y).
    \end{align*} Choosing $\alpha =e$, which maximizes $\log \alpha /\alpha$, then gives the result.
\end{proof}

Using Lemmas \ref{UpperBoundLemma} and \ref{lowerBoundLemma}, we can now prove the following proposition: \begin{prop}\label{propPrimesInAnnuli}
    Let $\mathbb{P}$ be the set of prime ideals in $\mathcal{O}_K$, and let $A(x,y)$ denote the set of ideals $\mathfrak{m}$ in $\mathcal{O}_K$ satisfying $x<N(\mathfrak{m})\leq y$. Then there are $x_0\geq 1$ and $\epsilon_0>0$ such that \begin{enumerate}
        \item[(i)] $|\mathbb{P}\cap A(16^x,16^{x+1})|\geq \frac{16^x}{x}$ for all $x\geq x_0$, and
        \item[(ii)]$|\mathbb{P}\cap A(16^x,16^{x+\epsilon})|\leq \frac{\sqrt{\epsilon}16^x}{x}$ for all $x\geq x_0$ and $\epsilon\in (0,\epsilon_0]$.
    \end{enumerate}
\end{prop}

\begin{proof}
    The upper bound (ii) is almost immediate by taking $\alpha=16^\epsilon$ in Lemma \ref{UpperBoundLemma}, so that one has \begin{align*}
        \sum_{16^x<N(\mathfrak{p})\leq 16^{x+\epsilon}}1&\leq \epsilon(\log 16)16^\epsilon\frac{16^x}{x\log 16}+O\Big(16^{x(1-\frac{1}{d})}x\Big) \\ &\leq 2\epsilon \frac{16^x}{x}+O\Big(16^{x(1-\frac{1}{d})}x\Big)\quad\quad (\epsilon\leq 1/4).
    \end{align*} Since $2\epsilon\leq \sqrt{\epsilon}$ for $\epsilon\leq 1/4$, and since the order of growth of $16^x/x$ is much larger than that of $16^{x(1-\frac{1}{d})}x$, the result follows upon taking $\epsilon_0=1/4$.\\

    The lower bound (i) follows from the fact that \begin{align*}|\mathbb{P}\cap A(16^x,16^{x+1})|&=|\mathbb{P}\cap A(0,16^{x+1})|-|\mathbb{P}\cap A(0,16^x)| \\ &= \pi_K(16^{x+1})-\pi_K(16^x)\end{align*} and the dyadic decomposition $(\frac{1}{2},16^x]=\bigcup_{k=0}^{4x}\Big(\frac{16^x}{2^{k+1}},\frac{16^x}{2^{k}}\Big]$ gives \begin{align*}
        \pi_K(16^x)&=\sum_{k=0}^{4x}\sum_{\frac{16^x}{2^{k+1}}< N(\mathfrak{p})\leq \frac{16^x}{2^{k}}}1\\ &\leq O_K(1)+ \sum_{k=0}^{4x-2}\Big[(2\log 2)\frac{16^x/2^{k+1}}{\log (16^x/2^{k+1})}+O_K\Big((16^x/2^{k+1})^{1-\frac{1}{d}}x\Big)\Big]\quad (\text{from Lemma \ref{UpperBoundLemma}}) \\ &= \sum_{k=0}^{4x-2}\frac{2^{4x-k}}{4x-k-1}+O_K\Big(16^{x(1-\frac{1}{d})}x\Big).
    \end{align*} Noting that the function $g(k):=\frac{2^{4x-k}}{4x-k-1}$ is decreasing for $k\in [0,4x-1-\frac{1}{\log 2}]$, we deduce \begin{align*}
        \sum_{k=0}^{4x-2}\frac{2^{4x-k}}{4x-k-1}&\leq \frac{2^{4x}}{4x-1}+\int_0^{4x-2}\frac{2^{4x-t}}{4x-t-1}\ dt+O(1) \\ &= \frac{2^{4x}}{4x-1}+2\int_1^{4x-1}\frac{2^u}{u}\ du+O(1).
    \end{align*} An application of De l'Hospital's rule provides that $\int_1^{4x-1}\frac{2^u}{u}\ du=(1+o(1))\frac{2^{4x-1}}{(4x-1)\log 2}$, and hence \begin{align*}
        \sum_{k=0}^{4x-2}\frac{2^{4x-k}}{4x-k-1}\leq \frac{2^{4x}}{4x-1}+(1+o(1))\frac{2^{4x}}{(4x-1)\log 2}=\frac{16^x}{x}\Big(\frac{1}{4}+\frac{1}{4\log 2}+o(1)\Big),
    \end{align*}
    
    so that $\pi_K(16^x)\leq (1+o(1))\frac{16^x}{x}+O_K(16^{x(1-\frac{1}{d})})$. One then has that, applying Lemma \ref{lowerBoundLemma}, \begin{align*}
        \pi_K(16^{x+1})-\pi_K(16^x)&\geq \frac{16^{x+1}}{e(x+1)\log 16}-(1+o(1))\frac{16^x}{x}+O_K(16^{x(1-\frac{1}{d})})\\ &= \Big[\frac{16}{e\log 16}-1+o(1)\Big]\cdot \frac{16^x}{x}+O_K(16^{x(1-\frac{1}{d})}) \\ &\geq \frac{16^x}{x}
    \end{align*} for sufficiently large $x$. This completes the proof of Proposition \ref{propPrimesInAnnuli}.
\end{proof}

\subsection{Combinatorial Ingredients}

We include two combinatorial lemmas, the first which can be easily extended from Lemma 3.4 in \cite{Richter_2021} to this setting with the help of Proposition \ref{propPrimesInAnnuli}, and the second which is precisely Lemma 3.5 in \cite{Richter_2021}:

\begin{lem}\label{combLemma1}
    Let $x_0$ be as in Proposition \ref{propPrimesInAnnuli}. There exists $\epsilon_1>0$ such that for all $\epsilon\in (0,\epsilon_1]$ and all $\delta\in (0,1)$ there exists $D=D(\epsilon,\delta)\in (0,1)$ with the following property: For all $n\geq x_0$ there are $x,y\in[n,n+1)$ with $\epsilon^4<y-x<\epsilon$ such that \begin{align*}
        |\mathbb{P}\cap A(16^x,16^{x+\delta})|\geq \frac{D 16^n}{n},\quad \text{and}\quad |\mathbb{P}\cap A(16^y,16^{y+\delta})|\geq \frac{D 16^n}{n}.
    \end{align*}
\end{lem}

\begin{lem}[Richter \cite{Richter_2021}, Lemma 3.5 ]\label{combLemma2}
    Fix $x_0\geq 1$ and $0<\epsilon<\frac{1}{2}$. Suppose $\mathcal{X}$ is a subset of $\mathbb{R}$ with the property that for every $n\geq x_0$, there exist $x,y\in \mathcal{X}\cap [n,n+1)$ with $\epsilon^4<y-x<\epsilon$. Let $k\geq \lceil 2/\epsilon^4\rceil$. Then for all $n_1,n_2,...,n_k\in \mathbb{N}_{\geq x_0}$ there exist $z,z_1,...,z_k\in \mathcal{X}$ such that \begin{enumerate}
        \item[(I)] $z_i\in[n_i,n_i+1)$ for all $1\leq i\leq k$
        \item[(II)] $z_1+...+z_k\in [z,z+\epsilon)$.
    \end{enumerate}
\end{lem}

We now have the tools necessary to prove Proposition \ref{prop2.2analogue}.

\begin{proof}[Proof of Proposition \ref{prop2.2analogue}] Let $\eta\in (0,1)$ be given. Let $x_0$ and $\epsilon_0$ be as in Proposition \ref{propPrimesInAnnuli}, and $\epsilon_1$ be as in Lemma \ref{combLemma1}. Pick any $\epsilon<\min\{\epsilon_0,\epsilon_1,\frac{\log(1+\eta)}{2\log 16}\}$. Set $N=\lceil 2/\epsilon^4\rceil$ and fix $k\geq N$. We will construct sets $S_1=S_1(\eta,k)$ and $S_2=S_2(\eta,k)$ as specified in Proposition \ref{prop2.2analogue}.

To do this, choose $\delta =\epsilon/k$, and let $D=D(\epsilon,\delta)$ be as in Lemma \ref{combLemma1}. Define \begin{align*}
    \mathcal{X}:=\Big\{x\geq N: |\mathbb{P}\cap A(16^x,16^{x+\delta})|\geq D 16^{\lfloor x\rfloor}/\lfloor x\rfloor\Big\}.
\end{align*} By Lemma \ref{combLemma1}, $\mathcal{X}$ is necessarily infinite and unbounded. For every $x\in \mathcal{X}$ define $P_x$ to be a set satisfying the following conditions: \begin{align}\label{setsOfPrimes}
    P_x\subset \mathbb{P}\cap A(16^x,16^{x+\delta})\text{ and }|P_x|=\lfloor D 16^{\lfloor x\rfloor}/\lfloor x\rfloor\rfloor
\end{align} We shall use the sets $P_x$ to construct the sets $S_1$ and $S_2$.\\

The set $\mathcal{X}$ satisfies the conditions of Lemma \ref{combLemma2}, which ensures, for each $k$-tuple $\mathbf{m}:=(m_1,...,m_k)\in \mathbb{N}_{\geq x_0}^k$, the existence of some $\zeta_{\mathbf{m}},\zeta_{1,\mathbf{m}},...,\zeta_{k,\mathbf{m}}\in \mathcal{X}$ that satisfy the conditions \begin{enumerate}
    \item[($\text{I}'$)] $\zeta_{i,\mathbf{m}}\in [m_i,m_i+1)$ for all $1\leq i\leq k$, and
    \item[($\text{II}'$)] $\zeta_{1,\mathbf{m}}+....+\zeta_{k,\mathbf{m}}\in [\zeta_{\mathbf{m}},\zeta_{\mathbf{m}}+\epsilon)$.
\end{enumerate} By definition of the sets $P_x$, one has \begin{align}\label{setsOfPrimesII}|P_{\zeta_i,\mathbf{m}}|=\lfloor D16^{m_i}/m_i\rfloor.\end{align} Now, let $M=M(D,\eta,k)$ be a constant that is to be determined later, and define sets $A_1,...,A_{k}\subset \mathbb{N}_{\geq x_0}$ inductively in the following manner:\begin{enumerate} 
\item[(1)] Pick some $s_1\in \mathbb{N}$ with $s_1>\max\{x_0,2k\}$ and let $A_1$ be a finite subset of $s_1\mathbb{N}=\{s_1n: n\in \mathbb{N}\}$ satisfying $\sum_{n\in A_1}1/n\geq M$.

\item[(2)] Given $A_i$, form $A_{i+1}$ by taking any $s_{i+1}>M+\max(A_1)+...+\max(A_i)$ and set $A_{i+1}$ to be a finite subset of $s_{i+1}\mathbb{N}$ satisfying $\sum_{n\in A_{i+1}}1/n\geq M$.
\end{enumerate}

We observe that these sets $A_i$ satisfy the following property: \begin{enumerate}
    \item[(A)] For any vector $\mathbf{q}=(q_1,...,q_{k})\in A_1\times ...\times A_{k}$ and $\mathbf{r}=(r_1,...,r_{k})\in A_1\times ...\times A_k$ with $\mathbf{q}\neq \mathbf{r}$, one has the distance between the integers $q_1+...+q_{k}$ and $r_1+...+r_{k}$ being at least $M$.
\end{enumerate} This follows because $\mathbf{q}\neq \mathbf{r}$ implies that there exists some $1\leq i_*\leq k$ satisfying $q_{i_*}\neq r_{i_*}$. Choose such an $i_*$ to be maximal, then one has \begin{align*}
    \Big|\sum_{i=1}^{k}(q_i-r_i)\Big| = \Big|\sum_{\substack{1\leq i\leq k\\ q_i\neq r_i}}(q_i-r_i)\Big| &= \Big|q_{i_*}-r_{i_*}+\sum_{\substack{1\leq i\leq k\\ q_i\neq r_i\\ i\neq i_*}}(q_i-r_i)\Big|\\ &\geq |q_{i_*}-r_{i_*}|-\Big|\sum_{\substack{1\leq i\leq k\\ q_i\neq r_i\\ i\neq i_*}}(q_i-r_i)\Big|\\  &\geq |q_{i_*}-r_{i_*}|-\sum_{\substack{1\leq i<i_*\\ q_i\neq r_i}}|q_i-r_i|,
\end{align*} with the last step following from the maximality assumption. On the one hand, since $q_{i_*}$ and $r_{i_*}$ are distinct elements of $s_{i_*}\mathbb{N}$, they satisfy $|q_{i_*}-r_{i_*}|\geq s_{i_*}$. On the other hand, if $1\leq i<i_*$, one has $|q_i-r_i|\leq k\leq \max A_i$ since both $q_i$ and $r_i$ are positive integers, and hence we have \begin{align*}
    \Big|\sum_{i=1}^{k}(q_i-r_i)\Big|\geq s_{i_*}-\sum_{1\leq i<i_*}\max A_i \geq M,
\end{align*} which provides the proof for property (A).\\

\begin{center}
    \textbf{The Construction}
\end{center}
We can now construct the collections of ideals $S_1$ and $S_2$ in 
Proposition \ref{prop2.2analogue}. Recalling the definitions of the sets $P_{\zeta_{i,\mathbf{m}}}$ as given in (\ref{setsOfPrimes}) and (\ref{setsOfPrimesII}), we form $S_2$ as the set of products \begin{align*}
    S_2:=\bigcup_{\mathbf{m}\in A_1\times \cdots\times A_k}P_{\zeta_{1,\mathbf{m}}}...P_{\zeta_{k,\mathbf{m}}}=\bigcup_{\mathbf{m}\in A_1\times \cdots\times A_k}\{\mathfrak{p}_1...\mathfrak{p}_k: \mathfrak{p}_i\in P_{\zeta_{i,\mathbf{m}}}\text{ for all }1\leq i\leq k\}.
\end{align*} We note that
\begin{align*}
|P_{\zeta_{1,\mathbf{m}}}...P_{\zeta_{k,\mathbf{m}}}|\leq \frac{D^k16^{m_1+\cdots+m_k}}{m_1\cdots m_k}\leq \frac{D^k16^{m_1+\cdots+m_k}}{m_1+\cdots+m_k}< \frac{D 16^{\lfloor \zeta_{\mathbf{m}}\rfloor}}{\lfloor \zeta_{\mathbf{m}}\rfloor}
\end{align*} where the last step follows from the fact $\zeta_{1,\mathbf{m}}+\cdots+\zeta_{k,\mathbf{m}}\geq \lfloor \zeta_{\mathbf{m}}\rfloor$ (as follows from ($\text{II}'$)), and because $D^{k}<D$. Hence we have by definition of $P_{\zeta_{\mathbf{m}}}$ that $|P_{\zeta_{1,\mathbf{m}}}\cdots P_{\zeta_{k,\mathbf{m}}}|\leq |P_{\zeta_\mathbf{m}}|$. In particular, there exists some collection of ideals $Q_{\mathbf{m}}\subseteq P_{\zeta_\mathbf{m}}$ with $|Q_{\mathbf{m}}|=|P_{\zeta_{1,\mathbf{m}}}\cdots P_{\zeta_{k,\mathbf{m}}}|$. For such a choice, we then form the collection of prime ideals \begin{align*}
S_1:=\bigcup_{\mathbf{m}\in A_1\times \cdots\times A_k}Q_\mathbf{m}.
\end{align*}

\begin{proof}[Proof that $S_1$ and $S_2$ satisfy (i)] By construction, if $\mathfrak{m}\in S_2$ then there exists an integer $k$-tuple $\mathbf{m}\in A_1\times \cdots\times A_k$ such that $\mathfrak{m}=\mathfrak{p}_{\zeta_{1,\mathbf{m}}}\cdots\mathfrak{p}_{\zeta_{k,\mathbf{m}}}$ with $\mathfrak{p}_{\zeta_{i,\mathbf{m}}}\in P_{\zeta_{i,\mathbf{m}}}$. Then $\Omega(\mathfrak{m})=k$, as desired. Whereas if $\mathfrak{p}\in S_1\subset P_{\zeta_\mathbf{m}}$ then $\mathfrak{p}$ is by definition a prime ideal. \end{proof}

\noindent \textit{Proof that $S_1$ and $S_2$ satisfy (ii).} By construction, we have $P_{\zeta_{i,\mathbf{m}}}\subset A(16^{\zeta_{i,\mathbf{m}}},16^{\zeta_{i,\mathbf{m}}+\delta})$, so that 
\begin{align*}
P_{\zeta_{1,\mathbf{m}}}\cdots P_{\zeta_{k,\mathbf{m}}}&\subset A(16^{\zeta_{1,\mathbf{m}}+ \cdots +\zeta_{k,\mathbf{m}}},16^{\zeta_{1,\mathbf{m}}+...+\zeta_{k,\mathbf{m}}+k\delta})\subset A(16^{\zeta_{\mathbf{m}}},16^{\zeta_{\mathbf{m}}+2\epsilon}),
\end{align*} with the last inclusion following from property ($\text{II}'$) and the definition of $\delta$ as $\epsilon/k$. Also, by definition we have $Q_\mathbf{m}\subset P_{\zeta_{\mathbf{m}}}\subset A(16^{\zeta_{\mathbf{m}}},16^{\zeta_{\mathbf{m}}+\delta})$. Since $\delta =\epsilon/k $ we clearly have $Q_\mathbf{m}\subset A(16^{\zeta_\mathbf{m}},16^{\zeta_\mathbf{m}+2\epsilon})$.

We first claim that each ideal in $S_2$ is distinct. By unique factorization of an ideal into prime ideals, for any fixed $\mathfrak{m}\in A_1\times...\times A_k$, the set $P_{\zeta_{1,\mathfrak{m}}}...P_{\zeta_{k,\mathfrak{m}}}=\{\mathfrak{p}_1...\mathfrak{p}_k:\mathfrak{p}_i\in P_{\zeta_{i,\mathfrak{m}}}\text{ for all }1\leq i\leq k\}$ satisfies $|P_{\zeta_{1,\mathfrak{m}}}...P_{\zeta_{k,\mathfrak{m}}}|=\prod_{i=1}^k|P_{\zeta_{i,\mathfrak{m}}}|$. Hence, it suffices to prove that the sets $P_{\zeta_{1,\mathbf{m_1}}}...P_{\zeta_{k,\mathbf{m_1}}}$  and $P_{\zeta_{1,\mathbf{m_2}}}...P_{\zeta_{k,\mathbf{m_2}}}$ have empty intersection for $\mathbf{m_1}\neq \mathbf{m_2}$, both in $A_1\times ...\times A_k$. Given some ideal $\mathfrak{u}\in P_{\zeta_{1,\mathbf{m}}}...P_{\zeta_{k,\mathbf{m}}}$, we can write $\mathfrak{u}=\mathfrak{p}_{\zeta_{1,\mathbf{m}}}...\mathfrak{p}_{\zeta_{k,\mathbf{m}}}$ for a collection of prime ideals $\mathfrak{p}_{\zeta_{i,\mathbf{m}}}$ with each $\mathfrak{p}_{\zeta_{i,\mathbf{m}}}\in P_{\zeta_{i,\mathbf{m}}}$. Since $\zeta_{i,\mathbf{m}}\in [m_i,m_i+1)$ for all $1\leq i\leq j$ by point ($\text{I}'$), and $\zeta_{1,\mathbf{m}}+...+\zeta_{k,\mathbf{m}}\in [\zeta_\mathbf{m},\zeta_\mathbf{m}+\epsilon)$ by point ($\text{II}'$), we then discern that \begin{align*}
   |m_1+...+m_k-\zeta_\mathbf{m}|&\leq \Big|-\zeta_\mathbf{m}+\sum_{i=1}^k\zeta_{i,\mathbf{m}}\Big|+\sum_{i=1}^k \Big|m_i-\zeta_{i,\mathbf{m}}\Big| <\epsilon+k,
\end{align*} which alongside property (A) with the choice of $M\geq 2k+3$, say, gives that $\mathfrak{u}$ cannot possibly be in any other collection $P_{\zeta_{1,\mathbf{m'}}}...P_{\zeta_{k,\mathbf{m'}}}$.

We next claim that each ideal in $S_1$ is distinct. We must prove that the sets $Q_\mathbf{m}$ are disjoint for different $\mathbf{m}\in A_1\times ...\times A_k$; since $Q_{\mathbf{m}}\subset P_{\zeta_\mathbf{m}}$, it suffices to prove that the $P_{\zeta_\mathbf{m}}$ are disjoint for different $\mathbf{m}\in A_1\times ...\times A_k$. But this is easily apparent, since $P_{\zeta_\mathbf{m}}\subset A(16^{\zeta_\mathbf{m}},16^{\zeta_\mathbf{m}+\delta})$ by definition, and alongside property (A) with the above choice of $M$ and the bound above gives that these annuli are disjoint. 

Regarding cardinality, it is immediately clear that $S_1$ and $S_2$ have the same cardinality since $|Q_\mathbf{m}|=|P_{\zeta_{1,\mathbf{m}}}...P_{\zeta_{k,\mathbf{m}}}|$, by construction. Also, since the sets $Q_\mathbf{m}$ and $P_{\zeta_{1,\mathbf{m}}}...P_{\zeta_{k,\mathbf{m}}}$ both lie in the annulus $A(16^{\zeta_\mathbf{m}},16^{\zeta_\mathbf{m}+2\epsilon})$, it is clear that the ratio of the norm of any element of the latter with the norm of any element of the former lies between $16^{-2\epsilon}$ and $16^{2\epsilon}$. Since $16^{2\epsilon}\leq 1+\eta$, it then follows that by enumerating $P_{\zeta_{1,\mathbf{m}}}...P_{\zeta_{k,\mathbf{m}}}=\{P_i:N(P_1)\leq...\leq N(P_r)\}$ and $Q_\mathbf{m}=\{Q_i:N(Q_1)\leq ...\leq N(Q_r)\}$ we have that $(1-\eta)N(P_i)\leq N(Q_i)\leq (1+\eta)N(P_i)$ for each $j=1,...,r$. Using the fact that each $\mathbf{m}$ corresponds to a distinct annulus $A(16^{\zeta_\mathbf{m}},16^{\zeta_\mathbf{m}+\delta})$, we can then extend this enumeration property to $S_1$ and $S_2$, so that for each $Q_i\in S_2$ there exists a $P_i\in S_1$ for which $(1-\eta)N(P_i)\leq N(Q_i)\leq (1+\eta)N(P_i)$.\\
\end{proof}

\begin{proof}[Proof that $S_1$ and $S_2$ satisfy (iii).] We first prove the coprimality condition for $S_1$, as it is simpler. We must show that $\E_{\mathfrak{p}\in S_1}^{\log}\E_{\mathfrak{p}'\in S_1}^{\log}\Phi(\mathfrak{p},\mathfrak{p}')\leq \eta$, for $\Phi(\mathfrak{p},\mathfrak{p}'):=N(\operatorname{gcd}(\mathfrak{p},\mathfrak{p}'))-1$. The set $S_1$ only contains prime ideals, so it is clear that $\Phi(\mathfrak{p},\mathfrak{p}')=0$ for $\mathfrak{p}\neq \mathfrak{p}'$; hence we have \begin{align}\label{expLogiii}
    \E_{\mathfrak{p}\in S_1}^{\log }\E_{\mathfrak{p}'\in S_1}^{\log}\Phi(\mathfrak{p},\mathfrak{p}')= \frac{\sum_{\mathfrak{p},\mathfrak{p}'\in S_1}\frac{\Phi(\mathfrak{p},\mathfrak{p}')}{N(\mathfrak{p})N(\mathfrak{p}')}}{\sum_{\mathfrak{p},\mathfrak{p}'\in S_1}\frac{1}{N(\mathfrak{p})N(\mathfrak{p}')}}\leq \frac{\sum_{\mathfrak{p}\in S_1}\frac{1}{N(\mathfrak{p})}}{\Big(\sum_{\mathfrak{p}\in S_1}\frac{1}{N(\mathfrak{p})}\Big)^2}=\frac{1}{\sum_{\mathfrak{p}\in S_1}\frac{1}{N(\mathfrak{p})}}.
\end{align} By definition of $S_1$ we evidently have $\sum_{\mathfrak{p}\in S_1}\frac{1}{N(\mathfrak{p})}=\sum_{\mathbf{m}\in A_1\times ...\times A_k}\sum_{P\in Q_\mathbf{m}}\frac{1}{N(\mathfrak{p})}$, and since $Q_\mathbf{m}\subset A(16^{\zeta_\mathbf{m}},16^{\zeta_\mathbf{m}+2\epsilon})\subset A(16^{\zeta_{\mathbf{m}}},16^{m_1+...+m_k+k+3\epsilon})$ one has $N(\mathfrak{p})\leq 16^{m_1+...+m_k+k+1}$ for each $\mathfrak{p}\in Q_\mathbf{m}$, so that \begin{align*}
    \sum_{P\in S_1}\frac{1}{N(\mathfrak{p})}\geq \sum_{\mathbf{m}\in A_1\times ... \times A_k}\frac{|Q_\mathbf{m}|}{16^{k+1}\cdot 16^{m_1+...+m_k}}.
\end{align*} Then since $|Q_\mathbf{m}|=|P_{\zeta_{1,\mathbf{m}}}...P_{\zeta_{k,\mathbf{m}}}|=\prod_{i=1}^k|P_{\zeta_{i,\mathbf{m}}}|$ one has $|Q_\mathbf{m}|\geq D^k16^{-k}16^{m_1+...+m_k}/(m_1m_2...m_k)$, so that \begin{align*}
    \sum_{\mathfrak{p}\in S_1}\frac{1}{N(\mathfrak{p})}\geq \sum_{\mathbf{m}\in A_1\times ... \times A_k}\frac{D^j}{16^{2k+1}m_1...m_k}\geq \frac{D^kM^k}{16^{2k+1}}.
\end{align*} Then if $M$ is chosen sufficiently large, then $\E_{\mathfrak{p}\in S_1}^{\log}\E_{\mathfrak{p}'\in S_1}^{\log}\Phi(\mathfrak{p},\mathfrak{p}')\leq \eta$ immediately follows.\\

Now, it suffices to show that $\E_{\mathfrak{m}\in S_2}^{\log}\E_{\mathfrak{m}'\in S_2}^{\log}\Phi(\mathfrak{m},\mathfrak{m}')\leq \eta$. By definition of the set $S_2$, one has \begin{align*}
    \sum_{\mathfrak{m},\mathfrak{m}'\in S_2}\frac{\Phi(\mathfrak{m},\mathfrak{m}')}{N(\mathfrak{m})N(\mathfrak{m}')}=\sum_{\mathbf{m},\mathbf{m'}\in A_1\times...\times A_k}\sum_{\mathfrak{m}\in P_{\zeta_{1,\mathbf{m}}}...P_{\zeta_{k,\mathbf{m}}}}\sum_{\mathfrak{m}'\in P_{\zeta_{1,\mathbf{m_2}}}...P_{\zeta_{k,\mathbf{m'}}}}\frac{\Phi(\mathfrak{m},\mathfrak{m}')}{N(\mathfrak{m})N(\mathfrak{m}')}.
\end{align*} If some $\mathfrak{m}\in P_{\zeta_{1,\mathbf{m}}}...P_{\zeta_{k,\mathbf{m}}}$ and $\mathfrak{m}'\in P_{\zeta_{1,\mathbf{m'}}}...P_{\zeta_{k,\mathbf{m'}}}$ are coprime then $\Phi(\mathfrak{m},\mathfrak{m}')=0$, so do not contribute towards the sum above. Whereas if $\mathfrak{m}$ and $\mathfrak{m}'$ are not coprime, set $\operatorname{gcd}(\mathfrak{m},\mathfrak{m}'):=\mathfrak{u}$ with $\mathfrak{u}\in \prod_{i\in F}P_{\zeta_{i,\mathbf{m'}}}$ for some nonempty minimal subset $F\subset \{1,...,k\}$  and $\mathfrak{u}'\in \prod_{i\not\in F}P_{\zeta_{i,\mathbf{m'}}}$ such that $\mathfrak{m}'=\mathfrak{uu}'$; this forces $m_i=m_i'$ for all $i\in F$ since $P_{\zeta_{i,\mathbf{m}}}$ and $P_{\zeta_{i,\mathbf{m'}}}$ are otherwise disjoint. One would then have \begin{align*}
    \frac{\Phi(\mathfrak{m},\mathfrak{m}')}{N(\mathfrak{m})N(\mathfrak{m}')}=\frac{N(\mathfrak{u})-1}{N(\mathfrak{m})N(\mathfrak{m}')}\leq \frac{1}{N(\mathfrak{m})N(\mathfrak{u}')}
\end{align*} and as a result \begin{align*}
    \sum_{\mathfrak{m},\mathfrak{m}'\in S_2}\frac{\Phi(\mathfrak{m},\mathfrak{m}')}{N(\mathfrak{m})N(\mathfrak{m}')}\leq \sum_{\substack{F\subset \{1,...,k\}\\ F\neq \varnothing}}\sum_{\substack{\mathbf{m},\mathbf{m'}\in A_1\times...\times A_k\\m_i=m_i',\ \forall i\in F}}\sum_{\mathfrak{m}\in P_{\zeta_{1,\mathbf{m}}}...P_{\zeta_{k,\mathbf{m}}}}\sum_{\mathfrak{u}'\in \prod_{i\not\in F}P_{\zeta_{i,\mathbf{m}'}}}\frac{1}{N(\mathfrak{m})N(\mathfrak{u}')}.
\end{align*} Then, using the fact that $P_{\zeta_{i,\mathbf{m}}}\subset A(16^{m_i},16^{m_i+1})$ and $P_{\zeta_{i,\mathbf{m'}}}\subset A(16^{m_i'},16^{m_i'+1})$ we then have \begin{align*}
    \sum_{\mathfrak{m}\in P_{\zeta_{1,\mathbf{m}}}...P_{\zeta_{k,\mathbf{m}}}}\sum_{\mathfrak{u}'\in \prod_{i\not\in F}P_{\zeta_{i,\mathbf{m}'}}}\frac{1}{N(\mathfrak{m})N(\mathfrak{u}')}&\leq \sum_{\mathfrak{m}\in P_{\zeta_{1,\mathbf{m}}}...P_{\zeta_{j,\mathbf{m}}}}\sum_{\mathfrak{u}'\in \prod_{i\not\in F}P_{\zeta_{i,\mathbf{m}'}}}\Big(\prod_{i=1}^k \frac{1}{16^{m_i}}\Big)\Big(\prod_{i\not\in F}\frac{1}{16^{m_i'}}\Big) \\ &= \Big(\prod_{i=1}^k\frac{|P_{\zeta_{i,\mathbf{m}}}|}{16^{m_i}}\Big)\Big(\prod_{i\not\in F}\frac{|P_{\zeta_{i,\mathbf{m'}}}|}{16^{m_i'}}\Big).
\end{align*} By definition, each set $P_{\zeta_{i,\mathbf{m}}}$ has cardinality $\lfloor D16^{m_i}/m_i\rfloor$, so that \begin{align*}
    \Big(\prod_{i=1}^k\frac{|P_{\zeta_{i,\mathbf{m}}}|}{16^{m_i}}\Big)\Big(\prod_{i\not\in F}\frac{|P_{\zeta_{i,\mathbf{m'}}}|}{16^{m_k'}}\Big)\leq \Big(\prod_{i=1}^k\frac{D}{m_i}\Big)\Big(\prod_{i\not\in F}\frac{D}{m_i'}\Big),
\end{align*} and hence we have \begin{align*}
    \sum_{\mathfrak{m},\mathfrak{m}'\in S_2}\frac{\Phi(\mathfrak{m},\mathfrak{m}')}{N(\mathfrak{m})N(\mathfrak{m}')}&\leq \sum_{\substack{F\subset \{1,...,k\}\\ F\neq \varnothing}}\sum_{\substack{\mathbf{m},\mathbf{m'}\in A_1\times...\times A_k\\m_i=m_i',\ \forall i\in F}}\Big(\prod_{i=1}^k\frac{D}{m_i}\Big)\Big(\prod_{i\not\in F}\frac{D}{m_i'}\Big) \\ &= \sum_{\substack{F\subset \{1,...,k\}\\ F\neq \varnothing}}D^{2k-|F|}\Big(\prod_{i=1}^k\sum_{m\in A_i}\frac{1}{m}\Big)\Big(\prod_{i\not\in F}\sum_{m\in A_i}\frac{1}{m}\Big).
\end{align*} One also has, since $\mathfrak{m}\in P_{\zeta_{1,\mathbf{m}}}...P_{\zeta_{k,\mathbf{m}}}$ implies that $N(\mathfrak{m})\leq 16^{m_1+...+m_k+k+1}$, that $\sum_{\mathfrak{m}\in S_2}\frac{1}{N(\mathfrak{m})}\geq \sum_{\mathbf{m}\in A_1\times ...\times A_k}\frac{|P_{\zeta_{1,\mathbf{m}}}...P_{\zeta_{j,\mathbf{m}}}|}{16^{m_1+...+m_k+k+1}} $; then since $|P_{\zeta_{1,\mathbf{m}}}...P_{\zeta_{k,\mathbf{m}}}|\geq D^j16^{-k}\frac{16^{m_1+...+m_k}}{m_1m_2...m_k}$ it follows that \begin{align*}
    \sum_{\mathfrak{m}\in S_2}\frac{1}{N(\mathfrak{m})}\geq \sum_{\mathbf{m}\in A_1\times ...\times A_k}\frac{D^k }{16^{2k+1}(m_1m_2...m_k)}\geq \frac{D^k}{16^{3k}}\prod_{i=1}^k\sum_{m\in A_i}\frac{1}{m}.
\end{align*} We then deduce that \begin{align*}
    \E_{\mathfrak{m}\in S_2}^{\log}\E_{\mathfrak{m}'\in S_2}^{\log}\Phi(\mathfrak{m},\mathfrak{m}')&=\frac{\sum_{\mathfrak{m},\mathfrak{m}'\in S_2}\frac{\Phi(\mathfrak{m},\mathfrak{m}')}{N(\mathfrak{m})N(\mathfrak{m}')}}{\sum_{\mathfrak{m},\mathfrak{m}'\in S_2}\frac{1}{N(\mathfrak{m})N(\mathfrak{m}')}} \\ &\leq \frac{\sum_{\substack{F\subset \{1,...,k\}\\ F\neq \varnothing}}D^{2k-|F|}\Big(\prod_{i=1}^k\sum_{m\in A_i}\frac{1}{m}\Big)\Big(\prod_{i\not\in F}\sum_{m\in A_i}\frac{1}{m}\Big)}{\Big(\frac{D^k}{16^{3k}}\prod_{i=1}^k\sum_{m\in A_i}\frac{1}{m}\Big)^2} \\ &\leq \sum_{\substack{F\subset \{1,...,k\}\\ F\neq \varnothing}}\frac{16^{6k}D^{-|F|}}{\Big(\prod_{i\in F}\sum_{m\in A_i}\frac{1}{m}\Big)}.
\end{align*} Finally, since $\sum_{m\in A_i}\frac{1}{m}\geq M$, we see that $\E_{\mathfrak{m}\in S_2}^{\log}\E_{\mathfrak{m}'\in S_2}^{\log}\Phi(\mathfrak{m},\mathfrak{m}')\leq \eta$ for sufficiently large $M$. This completes the proof of Proposition \ref{prop2.2analogue}.
\end{proof}

\section{Appendix A}
In this section, we use Abel summation and ideal density to provide general estimates for sums over ideals.

\begin{prop}\label{AppendixAProp}
    Let $K$ be a number field of degree $d$, and let $G(x)=\sum_{N(\mathfrak{m})\leq x}g[N(\mathfrak{m})]$ for a function $g\in C^1([1,\infty))$. Then one has \begin{align*}
        G(x)=c_Kg(1)\cdot xg(x)+c_K\int_1^xg(t)\ dt+c_Kg(1)+E,
    \end{align*} where \begin{align*}
        E\ll_K |g(x)|x^{1-\frac{1}{d}}+\int_1^x|g'(t)|t^{1-\frac{1}{d}}\ dt.
    \end{align*}
\end{prop}

\begin{proof}
    Let $\phi(n)$ denote the number of ideals with norm $n$ in $I_K$, then for any $x\in \mathbb{N}_0$, \begin{align*}
        G(x)=\sum_{N(\mathfrak{m})\leq x}g[N(\mathfrak{m})]=\sum_{n=1}^{x}\phi(n)g(n).
    \end{align*} Using partial summation, we can write this as \begin{align*}
        G(x)&=g(x)\sum_{n=1}^{x }\phi(n)-\int_1^xg'(t)\sum_{1\leq n\leq t}\phi(n)\ dt \\ &= g(x)\sum_{N(\mathfrak{m})\leq x}1-\int_1^xg'(t)\sum_{N(\mathfrak{m})\leq t}1\ dt \\ &=g(x)\Big(c_Kx+O_K(x^{1-\frac{1}{d}})\Big)-\int_1^xg'(t)\Big(c_Kt+O_K(t^{1-\frac{1}{d}})\Big)\ dt \\ &= c_Kxg(x)-c_K\int_1^xtg'(t)\ dt+O_K\Big(g(x)x^{1-\frac{1}{d}}+\int_1^x|g'(t)|t^{1-\frac{1}{d}}\ dt\Big) \\ &:= c_Kxg(x)-c_K\Big(xg(x)-g(1)-\int_1^xg(t)\ dt\Big)+E \\ &= c_Kg(1)\cdot xg(x)+c_K\int_1^xg(t)\ dt+c_Kg(1)+E,
    \end{align*} where \begin{align*}
        E\ll_K |g(x)|x^{1-\frac{1}{d}}+\int_1^x|g'(t)|t^{1-\frac{1}{d}}\ dt.
    \end{align*} 
\end{proof}

\section{Appendix B}

In this section we recall several well-known facts about the connection between equidistribution of a sequence and exponential sums. The most fundamental is Weyl's Criterion:

\begin{lem}[Weyl's Criterion]
    A sequence $(a_n)\subset \mathbb{R}$ is uniformly distributed modulo 1 if and only if for each $\ell\neq 0$, $\sum_{n\leq x}e(\ell a_n)=o(x)$.
\end{lem} In the positive integer setting, one also has an associated equidistribution criterion: \begin{lem}
    A sequence $(a_n)\subset \mathbb{Z}_{\geq 0}$ is equidistributed modulo $m$ if and only if $\sum_{n\leq x}e(\ell a_n/m)=o(x)$ for each $\ell \in \{1,...,m-1\}$.
\end{lem}
\begin{proof}
    After expanding out the square and applying Parseval's theorem over $\mathbb{Z}_m$, one observes that \begin{align*}\sum_{\ell \in \mathbb{Z}_m}\Big|-\frac{1}{m}+\frac{1}{x}\sum_{\substack{n\leq x\\ a_n\equiv \ell\ (\text{mod }m)}}1\Big|^2=\frac{1}{m}\sum_{\ell =1}^{m-1}\Big|\frac{1}{x}\sum_{n\leq x}e(\ell a_n/m)\Big|^2.\end{align*} From this it is not hard to deduce the lemma.
\end{proof}

\section{Appendix C}

In this section we show how $L(x)=o(x)$ implies that $M(x)=o(x)$, as defined in the Introduction.

\begin{lem}
    Let $K$ be an algebraic number field. Then $L(x)=o(x)$ implies that $M(x)=o(x)$.
\end{lem}

\begin{proof}
    We claim the following two statements are true (here $*$ denotes the Dirichlet convolution operator): \begin{enumerate}
    \item[(i)] Suppose that $\mathbf{1}_{\square}$ is the indicator function for the square ideals. Then if $\mathbf{1}_{\square}^{-1}$ denotes the Dirichlet inverse of $\mathbf{1}_{\square}$ --- i.e. the function such that $(\mathbf{1}_\square*\mathbf{1}_\square^{-1})(\mathfrak{m}):=\sum_{\mathfrak{d}_1\mathfrak{d}_2=\mathfrak{m}}\mathbf{1}_\square(\mathfrak{d}_1)\mathbf{1}_\square^{-1}(\mathfrak{d}_2)=\begin{cases}
        1 & \mathfrak{m}=K \\0 & \text{otherwise}
    \end{cases}$ --- then $\mathbf{1}_\square^{-1}$ is 1-bounded and $\text{supp }\mathbf{1}_\square\supseteq\text{supp }\mathbf{1}_\square^{-1}$.
    \item[(ii)] Let $\mu$ and $\lambda$ be the M\"{o}bius and Liouville functions over ideals in $I_K$, respectively. Let $g$ be the function such that $\mu=g*\lambda$. Then $g=\mathbf{1}_{\square}^{-1}$.
\end{enumerate} Let us explain how these two statements together imply the result. Writing $\mu=g*\lambda$, it is not hard to see then that \begin{align*}
M(x)=\sum_{N(\mathfrak{m})\leq x}g(\mathfrak{m})L\Big(\frac{x}{N(\mathfrak{m})}\Big).
\end{align*} If $g=(\mathbf{1}_{\square})^{-1}$ then by (i) $g$ is 1-bounded with support on square ideals, so that \begin{align*}
    \frac{M(x)}{x}=\frac{1}{x}\sum_{N(\mathfrak{m})^2\leq x}g(\mathfrak{m}^2)L\Big(\frac{x}{N(\mathfrak{m})^2}\Big)=\sum_{N(\mathfrak{m})^2\leq x}\frac{g(\mathfrak{m}^2)}{N(\mathfrak{m})^2}\cdot \frac{L(x/N(\mathfrak{m})^2)}{x/N(\mathfrak{m})^2}.
\end{align*} If $L(x)=o(x)$, then for each $\epsilon>0$ there exists some $C(\epsilon)$ for which $x\geq C(\epsilon)\implies |L(x)/x|<\epsilon$, and hence \begin{align*}
    \Big|\frac{M(x)}{x}\Big|&\leq \sum_{N(\mathfrak{m})^2\leq x}\frac{1}{N(\mathfrak{m})^2}\Big|\frac{L(x/N(\mathfrak{m})^2)}{x/N(\mathfrak{m})^2}\Big|\\ &=\sum_{N(\mathfrak{m})^2\leq x/C(\epsilon) }\frac{1}{N(\mathfrak{m})^2}\Big|\frac{L(x/N(\mathfrak{m})^2)}{x/N(\mathfrak{m})^2}\Big|+\sum_{x/C(\epsilon)<N(\mathfrak{m})^2\leq x}\frac{1}{N(\mathfrak{m})^2}\Big|\frac{L(x/N(\mathfrak{m})^2)}{x/N(\mathfrak{m})^2}\Big|\\ &< \epsilon \sum_{N(\mathfrak{m})^2\leq x/C(\epsilon)}\frac{1}{N(\mathfrak{m})^2}+\sum_{x/C(\epsilon)<N(\mathfrak{m})^2\leq x}\frac{1}{N(\mathfrak{m})^2}.
\end{align*} Then since $\sum_{\mathfrak{m}}\frac{1}{N(\mathfrak{m})^2}<\infty$ (as can be seen by Proposition \ref{AppendixAProp}), for sufficiently large $x$ we can make this arbitrarily small to show that $M(x)=o(x)$.\\

\textit{Proof of (i).} Since $\mathbf{1}_\square$ is multiplicative, one has, for $\mathfrak{m}=\prod_{i}\mathfrak{p}_i^{\alpha_i}$ that $\mathbf{1}_\square(\mathfrak{m})=\prod_i \mathbf{1}_{2|\alpha_i}$. By definition of Dirichlet inverse, for each $\alpha_i\geq 1$ one has $0=\sum_{\mathfrak{d}_1\mathfrak{d}_2=\mathfrak{p}_i^{\alpha_i}}\mathbf{1}_\square(\mathfrak{d}_1)f^{-1}(\mathfrak{d}_2)=\sum_{c=0}^{\alpha_i}\mathbf{1}_\square(\mathfrak{p}_i^c)\mathbf{1}_\square^{-1}(\mathfrak{p}_i^{\alpha_i-c})$. Then \begin{align*}
    0=\sum_{c=0}^{\lfloor \alpha_i/2\rfloor}\mathbf{1}_\square^{-1}(\mathfrak{p}_i^{\alpha_i-2c}).
\end{align*} At $\alpha_i=0$ one has $\mathbf{1}_\square^{-1}(1)=1$; at $\alpha_i=1$ one has $\mathbf{1}_\square^{-1}(\mathfrak{p}_i)=0$ by the above expression. Now let us induct on $\alpha_i$ to show that $\mathbf{1}_\square^{-1}(\mathfrak{p}_i^{\alpha_i})=0$ for $\alpha_i\equiv 1\ (\text{mod }2)$. The base case is shown; suppose that there exists some even $k$ for which $\mathbf{1}_\square^{-1}(\mathfrak{p}_i^{\alpha_i})=0$ for each $\alpha_i\leq k$ odd. One has $0=\sum_{c=0}^{\lfloor (k+1)/2\rfloor}\mathbf{1}_\square^{-1}(\mathfrak{p}_i^{k+1-2c})=\mathbf{1}_\square^{-1}(\mathfrak{p}_i^{k+1})$, so that $\mathbf{1}_\square^{-1}(\mathfrak{p}_i^{k+1})=0$ as well. This provides that $\text{supp }\mathbf{1}_\square^{-1}\subseteq \text{supp }\mathbf{1}_\square$, since $\mathbf{1}_\square^{-1}$ is multiplicative (in general, the Dirichlet inverse of a multiplicative function is multiplicative; see, e.g., \cite{apostolNT}, Thm. 2.16).

To show boundedness, we again use the above expression for prime power ideals, i.e. that $0=\sum_{c=0}^{\lfloor \alpha_i/2\rfloor}\mathbf{1}_\square^{-1}(\mathfrak{p}_i^{\alpha_i-2c})$. It is clear that $\mathbf{1}_\square(1)=\mathbf{1}_\square^{-1}(1)=1$, hence evaluating at $\alpha_i=2$ gives $\mathbf{1}_\square^{-1}(\mathfrak{p}_i^2)=-1$. Then one has $\mathbf{1}_\square^{-1}(\mathfrak{p}_i^4)=-\sum_{1\leq c\leq 2}\mathbf{1}_\square^{-1}(\mathfrak{p}_i^{\alpha_i-2c})=-(-1+1)=0$, and $\mathbf{1}_\square^{-1}(\mathfrak{p}_i^6)=-\sum_{1\leq c\leq 3}\mathbf{1}_\square^{-1}(\mathfrak{p}_i^{\alpha_i-2c})=-(0-1+1)=0$. It is not hard to see that $\mathbf{1}_\square^{-1}(\mathfrak{p}_i^{2k})=\begin{cases}
    1 & k = 0 \\
    -1 & k=1 \\
    0 & k\geq 2.
\end{cases}$ by induction; then since $\mathbf{1}_\square^{-1}$ is multiplicative, this provides a proof of boundedness since $|\mathbf{1}_\square^{-1}(\mathfrak{m})|=\prod_{\mathfrak{p}|\mathfrak{m}}|\mathbf{1}_\square^{-1}(\mathfrak{p}^{x_\mathfrak{p}(\mathfrak{m})})|\leq 1$.

\textit{Proof of (ii).} We first claim that $\lambda*1=\mathbf{1}_{\square}$. Indeed, it suffices to prove that $\mathbf{1}_{\square}*\mu=\lambda$. One has \begin{align*}
    (\mathbf{1}_{\square}*\mu)(\mathfrak{m})=\sum_{\mathfrak{u\mathfrak{d}^2=\mathfrak{m}}}\mu(\mathfrak{u}).
\end{align*} Write $\mathfrak{m}=\mathfrak{u}_0\mathfrak{d}_{max}^2$ for some $\mathfrak{u}_0$ squarefree, then the above is the same as saying \begin{align*}
    (\mathbf{1}_{\square}*\mu)(\mathfrak{m})=\sum_{\mathfrak{d}_1\mathfrak{d}_2=\mathfrak{d}_{max}}\mu(\mathfrak{u}_0\mathfrak{d}_2^2).
\end{align*} Each summand vanishes other than the case $\mathfrak{d}_2=1$, so that $(\mathbf{1}_{\square}*\mu)(\mathfrak{m})=\mu(\mathfrak{u}_0)=(-1)^{\Omega(\mathfrak{u}_0)}$. Then since $\Omega(\mathfrak{u}_0)\equiv \Omega(\mathfrak{m})\ (\text{mod }2)$, we have our desired equivalence.\\  Then convolving the equation $\lambda*1=\mathbf{1}_{\square}$ with $g$ would give $\lambda*g*1=g*\mathbf{1}_{\square}$, and since $\lambda*g=\mu$ by definition one would have $g*\mathbf{1}_{\square}=\mu*1=\delta$. But then $g=\mathbf{1}_{\square}^{-1}$, by definition and uniqueness of the Dirichlet inverse.
\end{proof}

\section*{Acknowledgements}

The author extends gratitude for the idea from Michael Lacey, and for discussions thereof. The author also thanks Alex Dunn and Vesselin Dimitrov for information regarding number fields, and Florian Richter and Vitaly Bergelson for their encouraging comments. Finally, the author thanks Rahul Sethi and Junzhe Mao for feedback on early drafts, and the anonymous referee for their many helpful comments.\\


\end{document}